\documentclass[11pt]{article}
\usepackage{amsmath}
\usepackage{amssymb}
\input{xy}
\xyoption{all}

\newtheorem{theorem}{Theorem}[section]
\newtheorem{proposition}[theorem]{Proposition}
\newtheorem{lemma}[theorem]{Lemma}
\newtheorem{corollary}[theorem]{Corollary}

\newtheorem{example}[theorem]{Example}

\newtheorem{remark}[theorem]{Remark}

\newenvironment{proof}{{\noindent\bf Proof.}}{\hfill $\Box$\par\vskip3mm}

\newcommand{\Ker}{{\rm Ker}\,}

\newcommand{\End}{{\rm End}}

\newcommand{\Ll}{\mathcal{L}}
\newcommand{\Mm}{\mathcal{M}}

\newcommand{\Rr}{\mathcal{R}}
\newcommand{\Tt}{\mathcal{T}}
\newcommand{\Ss}{\mathcal{S}}

\def\NN{{\mathbb N}}
\def\CC{{\mathbb C}}

\title{The Generating Condition for Coalgebras}
\author{\sc Miodrag Cristian Iovanov \small \\  Department of Algebra, University of Bucharest \\ Academiei 14, Bucharest, Romania \\and\\ State University of New York @ Buffalo, \\244 Mathematics Building, Buffalo NY, 14260}

\begin{document}
\baselineskip16pt
\maketitle

\date{}


\begin{abstract}
For a ring $R$, the properties of being (left) selfinjective or being cogenerator for the left $R$-modules do not imply one another, and the two combined give rise to the important notion of PF-rings. For a coalgebra $C$, (left) self-projectivity implies that $C$ is generator for right comodules and the coalgebras with this property were called right quasi-co-Frobenius; however, whether the converse implication is true is an open question. We provide an extensive study of this problem. We show that this implication does not hold, by giving a large class of examples of coalgebras having the "generating property". In fact, we show that any coalgebra $C$ can be embedded in a coalgebra $C_\infty$ that generates its right comodules, and if $C$ is local over an algebraically closed field, then $C_\infty$ can be chosen local as well. We also give some general conditions under which the implication "$C$-projective (left) $\Rightarrow C$ generator for right comodules" does work, and such conditions are when $C$ is right semiperfect or when $C$ has finite coradical filtration.
\footnote{The author was partially supported by the contract nr. 24/28.09.07 with UEFISCU "Groups, quantum groups, corings and representation theory" of CNCIS, PN II (ID\_1002)
}
\end{abstract}


\section{Introduction}

Let $R$ be a ring or an algebra. There are two very basic properties of the ring, which are very important for the theory of rings and modules: a homological one, that $R$ is projective as left or right module and a categorical one, the fact that $R$ generates all its left (and right) modules. The dual properties, namely when is $R$ injective as left or right $R$-module on one hand (i.e. $R$ is selfinjective), and when does $R$ cogenerate its left or its right $R$-modules, have been the subject of much study in ring theory (see for example \cite{F2}, 4.20-4.23, 3.5 and references therein). The rings (algebras) that satisfy both conditions are the same as the PF-rings (pseudo-Frobenius), which are rings $R$ such that every faithful right $R$-module is a generator. There are many known equivalent characterizations of these rings as well as many connections of these rings with other notions, such as the QF-rings (quasi-Frobenius rings=artinian rings with annihilator duality between the left and right ideals, equivalently left and right artinian, cogenerator and self-injective rings), semiperfect rings, perfect rings, FPF rings or Frobenius algebras. They have been introduced as generalizations of Frobenius algebras, and they retain much of the module (representation) theoretic properties of these algebras. The following theorem recalls some equivalent characterizations of PF-rings (see \cite[4.20]{F2}) and for QF-rings (see also \cite{CR}):

\begin{theorem}
(1)$R$ is right PF if and only if it satisfies either one of the following conditions:\\
(i) $R$ is an injective cogenerator.\\
(ii) $R=\bigoplus\limits_{i=1}^n e_iR$ with $e_i^2=e_2$ and $e_iR$ is indecomposable injective with simple socle for all $i$.\\
(2) $R$ is a QF-ring if and only if every injective right $R$-module is projective and if and only if every injective left $R$-module is projective.
\end{theorem}

Dually, the analogue questions have be raised in the case of coalgebras and comodules over coalgebra. We will refer to \cite{A}, \cite{DNR}, \cite{M} or \cite{S} for various basic definitions and results in the theory coalgebras and their comodules. Recall that a coalgebra over a field $K$ is defined by dualizing the categorical diagrams defining the notion of algebra. That is, a coalgebra is an algebra in the category dual to that of $K$-vector spaces. Thus, a coalgebra $(C,\Delta,\varepsilon)$ has a comultiplication $\Delta:C\rightarrow C\otimes C$ and counit $\varepsilon:C\rightarrow K$ satisfying appropriate co-associativity and co-unit relations. We use Sweedler's sigma notation which writes $\Delta(c)=\sum\limits_{(c)}c_1\otimes c_2$ or the simplified notation with the summation symbol omitted $\Delta(c)=c_1\otimes c_2\in C\otimes C$, and this will always be understood as a summation rather then a single tensor monomial. Then the coassociativity of $C$ writes $c_1\otimes c_{21}\otimes c_{22}=c_{11}\otimes c_{12}\otimes c_2$ and the counit property $\varepsilon(c_1)c_2=c=c_1\varepsilon(c_2)$. A right $C$-comodule $(M,\rho)$ is defined as a vector space with a comultiplication $\rho:M\rightarrow M\otimes C$ and satisfying appropriate compatibility conditions; using a similar convention $\rho(m)=m_0\otimes m_1$, these conditions write $m_{00}\otimes m_{01}\otimes m_1=m_0\otimes m_{11}\otimes m_{12}$ and $m_0\varepsilon(m_1)=m$. The category of right $C$-comodules is be denoted $\Mm^C$, and that of the similarly defined left comodules is ${}^C\Mm$. In analogy, we will use the notation ${}_R\Mm$ and $\Mm_R$ of the categories of left, respectively right $R$-modules over a ring $R$. The dual $C^*$ of a coalgebra is an algebra with the convolution product $(fg)(c)=(f*g)(c)=f(c_1)g(c_2)$ and any right $C$-comodule $M$ is also a left $C^*$-module by the action $c^*\cdot m=c^*(m_1)m_0$, where $c^*\in C^*$, $m\in M$ and $\rho(m)=m_0\otimes m_1\in M\otimes C$.   

\vspace{.4cm}

A coalgebra $C$ over a field $K$ is always a cogenerator for its comodules and is also injective as a comodule over itself. The dual properties in the coalgebra situation, corresponding to the selfinjectivity and the cogenerator properties of a ring (or an algebra), are that of a coalgebra being projective as a right (or left) comodule or being a generator for the right (or left) comodules. These conditions were studied for coalgebras in \cite{NT1} and \cite{NT2}, where QcF (quasi-co-Frobenius) coalgebras were introduced as the dualization of QF-algebras and in some respects of PF rings. It is proved there that 
 
\begin{theorem}\label{t0} 
The following assertions are equivalent for a coalgebra $C$.\\
(i) $C$ embeds in a direct sum of copies of $C^*$ as left $C^*$-modules.\\
(ii) $C$ is a torsionless left $C^*$-modules i.e. $C$ embeds in a direct product of copies of $C^*$.\\
(iii) Every injective right $C$-comodule is projective.\\
(iv) $C$ is a projective right $C$-comodule.\\
(v) $C$ is a projective left $C^*$-module.
\end{theorem}

A coalgebra satisfying these equivalent condition is called left QcF. Moreover, if these hold, then $C$ is a generator in ${}^C\Mm$, the category of left $C$-comodules. This concept is not left-right symmetric, unlike the algebra counterpart, the QF-algebras (see \cite[Example 3.3.7]{DNR} and \cite[Example 1.6]{NT1}). It is shown in \cite{NT2} (see also \cite[Theorem 3.3.11]{DNR}) that a coalgebra is left and right QcF if and only if $C^C$ generates right comodules and is projective as right comodule, equivalently ${}^CC$ is a projective generator in ${}^C\Mm$ and these are further equivalent to $C$ being a generator for both $\Mm^C$ and ${}^C\Mm$, characterizations that dualize known characterizations of finite dimensional QF algebras. Other symmetric characterizations which also generalize the characterizations of Frobenius algebras and co-Frobenius coalgebras \cite{I} and strongly motivate the consideration of QF algebras and QcF coalgebras as the generalization of Frobenius algebras are given in \cite{I1}. Though, it remained open whether the fact that $C$ is a generator for ${}^C\Mm$ is actually enough to imply the fact that $C$ is left QcF, i.e. if it implies that $C$ is projective as left $C^*$-module. In fact, the question has been has been studied very recently in \cite{NTvO}, where some partial results are given. Among these, it is shown that the answer to this question is positive in the case $C$ has finite coradical filtration. The general question however is left as an open question.

\vspace{.4cm}

In the case of a ring $R$, there is no implication between the property of being left self-injective and that of $R$ being left cogenerator. An example of a ring $R$ which is a non-injective cogenerator in $M_R$ is the $K$-algebra with basis ${1}\cup\{e_i\mid i=0,1,2\dots\}\cup\{x_i\mid i=0,1,2\dots\}$ with identity $1$ and with $e_ix_j=\delta_{i,j}x_j$, $x_je_i=\delta_{i,i-1}x_j$, $e_ie_j=\delta_{i,j}e_i$ and $x_ix_j=0$ for all $i,j$ - see \cite[24.34.2, p. 215]{F1}. Conversely, a ring $R$ which is a right cogenerator, is right selfinjective if and only if it is semilocal (see again \cite[24.10-24.11]{F1}), and there are selfinjective rings which are not semilocal, and thus they are not right cogenerators. Such an example can even be obtained as a profinite algebra, that is, an algebra which is the dual of a coalgebra - see Example \ref{e2}.

\vspace{.4cm}

We will say that a coalgebra has the right generating condition if it generates all its right comodules. There are two main results in this paper. Firstly, we examine some conditions under which the right generating condition of a coalgebra implies the fact that $C$ is right QcF (projective as right $C^*$-module). Among these, we consider three important conditions in the theory of coalgebras: semiperfect coalgebras, coalgebras of finite coradical filtration and coalgebras of finite dimensional coradical (almost connected). We show that 

\vspace{.2cm}

($*$) \emph{a coalgebra with the right generating condition and whose indecompusable (injective) left components are of finite Loewy length is necessarily right QcF (the converse is known to hold).} 

\vspace{.2cm}

Therefore, for a coalgebra $C$ with the right generating condition, the above is an equivalence, and the coalgebra $C$ being QcF is further equivalent to $C$ being right semiperfect (see \cite{L}). As a consequence, we see that implication ($*$) holds whenever the coalgebra has finite coradical filtration, and this allows us to reobtain the main result of \cite{NTvO} in a direct short way. Secondly, we show that every coalgebra $C$ embeds in a coalgebra $C_\infty$ that has the right generating condition (in fact, $C_\infty$ will even have any of its finite dimensional right comodules as a quotient). Thus, starting with a coalgebra $C$ which is not right semiperfect, we will get a coalgebra $C_\infty$ which is not right semiperfect (see \cite{L}) and thus, by well known properties of coalgebras, $C_\infty$ will not be right QcF. Moreover, if we start with a connected coalgebra (coalgebra having the coradical of dimension 1) over an algebraically closed field, we show that the coalgebra $C_\infty$ can be constructed to be local as well, therefore showing that the third mentioned condition for coalgebras - the coalgebra having finite dimensional coradical - is not enough for the right generating condition to imply the QcF property. 


\section{Loewy series and the Loewy length of modules}

We first recall a few well known facts on the Loewy series of modules. Let $M$ be a module over a ring $R$. We denote $L_0(M)=0$, $L_1(M)=s(M)$ - the socle of $M$, the sum of all the simple submodules of $M$. The Loewy series of $M$ is defined inductively as follows: if $L_n(M)$ is defined, $L_{n+1}(M)$ is such that $L_{n+1}(M)/L_n(M)$ is the socle of $M/L_n(M)$. More generally, if $\alpha$ is an ordinal, and $(L_\beta)_{\beta<\alpha}$ were defined, then\\
$\bullet$ if $\alpha=\beta+1$ is a successor, then one defines $L_{\beta+1}(M)$ such that $L_{\beta+1}(M)/L_\beta(M)=s(M/L_\beta(M))$;\\
$\bullet$ if $\alpha$ is a limit (i.e. not a successor) then one defines $M_\alpha=\bigcup\limits_{\beta<\alpha}M_\beta$.\\
If for some $\alpha$, $M=M_\alpha$ we say that $M$ has its Loewy length defined and the least ordinal $\alpha$ with this property will be called the Loewy length of $M$; we will write $lw(M)=\alpha$. It is known that modules having the Loewy length defined are exactly the semiartinian modules, that is, modules $M$ such that $s(M/N)\neq 0$ for any submodule $N$ of $M$ with $N\neq M$. We refer to \cite{N} as a good source for these facts.\\
We also recall a few well known facts on the Loewy length of modules. Throughout this paper, only modules of finite Loewy length will be used; however these properties hold in general for all modules. In the following, whenever we write $lw(M)$ we understand that this implicitly also means the Loewy length of $M$ is defined (and for our purposes, it will also be enough to assume that $lw(M)$ is finite). 

\begin{proposition}\label{1}
For any ordinal $\alpha$ (or $\alpha$ non-negative integer) we have:\\
(i) If $N$ is a submodule of $M$ then $L_\alpha(N)\leq L_\alpha(M)$ and in fact $L_\alpha(N)=N\cap L_\alpha(M)$.\\
(ii) If $f:N\rightarrow M$ is a morphism of modules, then $f(L_\alpha(N))\subseteq L_\alpha(M)$.\\
(iii) If $N$ is a submodule of $M$ then $lw(N)\leq lw(M)$, $lw(M/N)\leq lw(M)$ and $lw(M)\leq lw(N)+lw(M/N)$.\\
(iv) $L_\alpha(\bigoplus\limits_{i\in I}M_i)=\bigoplus\limits_{i\in I}L_\alpha(M_i)$ and $lw(\bigoplus\limits_{i\in I}M_i)=\sup\limits_{i\in I}\,lw(M_i)$.
\end{proposition}

Let $(C,\Delta,\varepsilon)$ be a coalgebra over an arbitrary field $K$, $A=C^*$ and let $M$ be a right $C$-comodule with comultiplication $\rho:M\rightarrow M\otimes C$. It is well known that $M$ has its Loewy length defined and in fact $lw(M)\leq \omega_0$, the first infinite ordinal. The coradical filtration of $C$ is defined by $C_0=L_1(C)$, ... $C_n=L_{n+1}(C)$. Let $J=J(A)$, the Jacobson radical of $A$; by \cite[Proposition 2.5.3]{DNR} we have $(J^n)^\perp=C_{n-1}$, where for $I<A$, $I^\perp=\{c\in C|f(c)=0,\,\forall f\in I\}$ and for $X\subseteq C$, $X^\perp=\{c^*\in C^*|c^*(x)=0,\,\forall x\in X\}$. Also, if $M$ is a right $C$-comodule (so a left $A$-module), $M^*$ becomes a right $A$-module in the usual way by the "dual action" $(m^*\cdot a)(m)=m^*(am)$, $m^*\in M^*,\,m\in M,\,a\in A$. The following Lemma gives the connection between the Loewy length of $M$ and $M^*$ and also provides a way to compute it for comodules of finite Loewy length.

\begin{lemma}\label{2}
Let $(M,\rho)$ be a right $C$-comodule. Then the following are equivalent:\\
(i) $J^n\cdot M=0$.\\
(ii) $M^*\cdot J^n=0$.\\
(iii) ${\rm Im} \rho\subseteq M\otimes C_{n-1}$.\\
(iv) $lw(M)\leq n$.
\end{lemma}
\begin{proof}
(i)$\Rightarrow$(ii) is straightforward.\\
(ii)$\Rightarrow$(i) If $f\in J^n$ and $m\in M$, then for all $m^*\in M^*$, $0=(m^*\cdot f)(m)=m^*(f\cdot m)$. Since this is true for all $m^*\in M^*$, we get $f\cdot m=0$.\\
(iii)$\Leftrightarrow$(iv) The map $\rho:M\rightarrow M\otimes C$ is a morphism of $C$-comodules by the coassociativity property. Moreover, $M\otimes C\simeq \bigoplus\limits_{i\in I}C$, where $I$ is a $K$-basis of $M$. Since $C_{n-1}=L_n(C)$, using these isomorphisms, we have that $L_n(\bigoplus\limits_{i\in I}C)=\bigoplus\limits_{i\in I}C_{n-1}$ and so $L_n(M\otimes C)=M\otimes C_{n-1}$. Therefore, if (iii) holds, since $\rho$ is also injective (by the counit property) we get $lw(M)=lw(\rho(M))\leq lw(M\otimes C_{n-1})=n$; conversely, if (iv) holds, then $M=L_k(M)$ for some $k\leq n$ so $\rho(M)\subseteq L_k(M\otimes C)\subseteq L_n(M\otimes C)=M\otimes C_{n-1}$.\\
(i)$\Rightarrow$(iii) For $m\in M$, let $\rho(m)=\sum\limits_{i=1}^km_i\otimes c_i\in M\otimes C$ and by a standard linear algebra observation we can choose the $m_i$'s to be linearly independent. For all $f\in J^n$, $0=f\cdot m=\sum\limits_{i=1}^kf(c_i)m_i$ and thus $f(c_i)=0,\,\forall i$, i.e. $c_i\in (J^n)^\perp=C_{n-1}$ for all $i=\overline{1,k}$. \\
(iii)$\Rightarrow$(i) is true, since $J^n\subseteq (J^n{}^\perp)^\perp=C_{n-1}^\perp$.
\end{proof}

Since the dual of a finite dimensional right $C$-comodule is a finite dimensional left $C$-comodule, we have

\begin{corollary}\label{3}
If $M$ is a finite dimensional right $C$-comodule (rational left $C^*$-module), then $M^*\in {}^C\Mm$ and $lw(M)=lw(M^*)$.
\end{corollary}

\section{The generating condition}


Let $\Ss$ (respectively $\Tt$) denote a system of representatives of simple left (respectively right) $C$-comodules. Then $C\simeq \bigoplus\limits_{S\in \Ss}E(S)^{n(S)}$ as left $C$-comodules, with $n(S)$ positive integers and $E(S)$ the injective envelopes of the comodule $S$. Similarly, $C=\bigoplus\limits_{T\in\Tt}E(T)^{p(T)}$ as right $C$-comodules. The we obviously have that $C$ generates all the right $C$-comodules if and only if $(E(T))_{T\in \Tt}$ is a system of generators. Recall that a coalgebra $C$ is right (left) semiperfect if and only if the $E(S)$'s are finite dimensional for all $S\in \Ss$ (resp. the $E(T)$'s $T\in\Tt$ are finite dimensional; see \cite{L} or \cite[Chapter 3]{DNR}). We first give a simple proposition that explains what is the property that coalgebras with the generating condition are missing to be QcF. 

\begin{proposition}\label{3a}
Let $C$ be a coalgebra. Then $C$ is left QcF if and only if $C$ is left semiperfect and generates its left comodules.
\end{proposition}
\begin{proof}
"$\Rightarrow$" is already known (see \cite[Chapter 3]{DNR}).\\
"$\Leftarrow$" It is well known that we have $C=\bigoplus\limits_{i\in I}E(T_i)$ a direct sum of right comodules (left $C^*$-modules), where $T_i$ are simple comodules, $C_0=\bigoplus\limits_{i\in I}T_i$ is the coradical of $C$ and $E(T_i)$ are injective envelopes of $T_i$ contained in $C$. For each $T_i$, $E(T_i)$ is finite dimensional, so $E(T_i)^*$ is a finite dimensional right $C^*$-modules which is rational, that is, it has a left $C$-comodule structure. Then there is an epimorphism of right $C^*$-modules $\phi_i:C^{n_i}\rightarrow E(T_i)^*\rightarrow 0$, where $n_i$ can be taken to be a (finite) number since $E(T_i)$ is finite dimensional. By duality, this gives rise to a morphism $\psi_i:E(T_i)\simeq (E(T_i)^*)^*\rightarrow (C^*)^{n_i}$ (given by $\psi_i(x)(c)=\phi_i(c)(x)$). Since $\phi_i$ is a surjective morphism of right $C^*$modules, it is easy to see that $\psi_i$ is an injective morphism of left $C^*$-modules. We then get a monomorphism of left $C^*$-modules $\bigoplus\limits_{i\in I}\psi_i:\bigoplus\limits_{i\in I}E(T_i)\hookrightarrow \bigoplus\limits_{i\in I}(C^*)^{n_i}$, a coproduct power of $C^*$, so $C$ is left QcF.
\end{proof}

The next proposition will be the key step in proving the main results of this section.

\begin{proposition}\label{4}
Suppose $C$ generates $\Mm^C$. If $S\in \Ss$ is such that $lw(E(S))=n$, then for each finite dimensional subcomodule $N$ of $E(S)$ with $lw(N)=n$, there is $T\in\Tt$ such that $N\simeq E(T)^*$.
\end{proposition}
\begin{proof}
Note that since $N$ has simple socle, $N^*$ is a right $C$-comodule which is local, say with a unique maximal subcomodule $X$. This is due to the duality $X\mapsto X^*$ between finite dimensional left and finite dimensional right $C$-comodules. Let $\bigoplus\limits_{i\in I}E(T_i)\stackrel{\varphi}{\rightarrow}N^*\rightarrow 0$ be an epimorphism in $M^C$; then $\exists\,i\in I$ such that $\varphi(E(T_i))\subsetneq X$, and then (for example by Nakayama lemma) we have $\varphi(E(T_i))=N^*$. Put $T=T_i$. We have a diagram of left $A$-modules
$$\xymatrix{
& E(S)^*\ar[d]^r\ar@{.>}[dl]_p & \\
E(T)\ar[r]_\varphi & N^* \ar[r] & 0
}$$
which is completed commutatively by a morphism $p$, since $E(S)^*$ is a direct summand in $C^*=A$ (the vertical map is the natural one). Let $P={\rm Im}(p)$; by (the left hand side version of) Lemma \ref{2}, $J^n\cdot E(S)^*=0$ and so $J^n\cdot P=p(J^n\cdot E(S)^*)=0$. But $P$ is finitely generated  (even cyclic, since $E(S)^*$ is so), and $P$ is also a right $C$-comodule (rational left $C^*$-module), and therefore it is finite dimensional. Thus its Loewy length is defined and $lw(P)\leq n$ by the same Lemma. Also, $\varphi\vert_P$ is injective. Indeed, otherwise $T\subseteq \ker \varphi \cap P=\ker(\varphi\vert_P)(\neq 0)$, since $T$ is essential in $E(T)$. Then $T=L_1(E(T))=L_1(P)$ and so $lw(P/T)=lw(P/L_1(P))<lw(P)\leq n$ (by the definition of Loewy length). But $\varphi$ factors to $\varphi:P/T\rightarrow N^*$ and therefore, using also Corollary \ref{3}, $lw(P/T)\geq lw(N^*)=n$ - a contradiction.\\
Since $\varphi\circ p=r$ is surjective, $\varphi\vert_P$ is an isomorphism with inverse $\theta$. This shows that the inclusion $\iota:P\hookrightarrow E(T)$ splits off ($\theta\circ\varphi\circ\iota=\theta\circ\varphi_P={\rm id}_p$) and since $E(T)$ is indecomposable, $P=E(T)$. Hence $\varphi$ is an isomorphism and $E(T)\simeq N^*$, so $N\simeq E(T)^*$ since they are finite dimensional.
\end{proof}

\begin{proposition}\label{5}
Suppose $C$ satisfies the right generating condition. Then for each $S\in\Ss$ such that $E(S)$ has finite Loewy length, there exists $T\in \Tt$ such that $E(S)\simeq E(T)^*$ and $E(S)$ is finite dimensional.
\end{proposition}
\begin{proof}
Let $n=lw(E(S))$. First note that there exists at least one finite dimensional subcomodule $N$ of $E(S)$ such that $lw(N)=n$: take $x\in L_n(E(S))\setminus L_{n-1}(E(S))$ and put $N=x\cdot C^*$ the left subcomodule (equivalently, right $A$-submodule) generated by $x$. Then $L_{n-1}(N)\neq N$ since otherwise $N\subseteq L_{n-1}(E(S))$, and therefore $n\leq lw(N)\leq lw(E(S))=n$. Let $N_0=N$. Assuming $E(S)$ is not finite dimensional, we can inductively build the sequence $(N_k)_{k\geq 0}$ of finite dimensional subcomodules of $E(S)$ such that $N_{k}/N_{k-1}$ is simple for all $k\geq 1$ (simple comodules are finite dimensional). Applying Proposition \ref{4} we see that each $N_k$ is local, since each is the dual of a comodule with simple socle (same argument as above; this also follows from the more general \cite[Lemma 1.4]{I}). Then $N_k/N_0$ has a composition series 
$$0=N_0/N_0\subseteq N_1/N_0\subseteq N_2/N_0\subseteq\dots\subseteq N_{k-1}/N_0\subseteq N_k/N_0$$
with each term of the series being local. Then, by duality, $M_k=(N_k/N_0)^*$ has a composition series 
$$0=X_0\subseteq X_1\subseteq X_2\subseteq\dots\subseteq X_{k-1}\subseteq X_k=(N_k/N_0)^*$$
such that $X_i\simeq (N_k/N_i)^*$, because of the short exact sequences of left $C^*$-modules and right $C$-comodules $0\rightarrow(\frac{N_k}{N_i})^*\rightarrow(\frac{N_k}{N_0})^*\rightarrow(\frac{N_i}{N_0})^*\rightarrow 0$. Therefore, $M_k/X_i\simeq (N_i/N_0)^*$ has simple socle (by duality), since $N_i/N_0$ are all local. Therefore, by definition, the above series of $M_k$ is the Loewy series and so $lw(N_k/N_0)=lw(M_k)=k$. But then $k=lw(N_k/N_0)\leq lw(N_k)\leq lw(E(S))=n$ for all $k$, which is absurd. Therefore $E(S)$ is finite dimensional. This also shows that the sequence $(N_k)_{k\geq 0}$ must terminate with some $N_k=E(S)$, because it can be continued whenever $N_k\neq E(S)$. Since $N_k\simeq E(T)^*$ for some $T\in\Tt$ by Proposition \ref{4}, this ends the proof.
\end{proof}


\begin{theorem}\label{t1}
Let $C$ be a coalgebra satisfying the right generating condition. Then the following conditions are equivalent: \\
(i) The injective envelope (as comodules) of every simple left comodule has finite Loewy length. \\
(ii) $C$ is right semiperfect.\\
(iii) $C$ is right QcF. \\
These conditions hold in pardicular if $C=C_n$ for some $n$, i.e. $C$ has finite coradical filtration.
\end{theorem}
\begin{proof}
We note that (iii)$\Rightarrow$(ii)$\Rightarrow$(i) are obvious so we only need to prove (i)$\Rightarrow$(iii). By Proposition \ref{5}, $\forall S\in\Ss,\,\exists T\in\Tt\,{\rm s.t.}\,E(S)\simeq E(T)^*$, so each $E(S)$ is projective as right $C^*$-module and it also embeds in $C^*$. Therefore, $C\simeq\bigoplus\limits_{S\in\Ss}E(S)^{n(S)}$ is projective as right $C^*$-module (and then also as left $C$-comodule). It also follows that since each $E(S)$ embeds in $C^*$ (it is actually a direct summand), we have an embedding $C\simeq\bigoplus\limits_{S\in\Ss}E(S)^{n(S)}\hookrightarrow \bigoplus\limits_{S\in\Ss}C^*{}^{n(S)}$.
\end{proof}

Note that the above provide another proof for Proposition \ref{3a}. In particular, it provides a direct proof for \cite[Theorem 4.1]{NTvO}. We also note that the property coming up in the above proofs, that the dual of every left indecomposable injective $C$-comodule is the dual of a right indecomposable injective, is proved to be equivalent to the coalgebra $C$ being QcF in \cite{I1}, \cite{I2}. We prefer giving the direct argument here.

\section{A general class of examples}

In this section we construct the general examples of this paper. The first goal is to start with an arbitrary coalgebra $C$ and build a coalgebra $D$ such that $C\subseteq D$ and $D$ satisfies the right generating condition.\\
Let $(C,\Delta,\varepsilon)$ be a coalgebra and $(M,\rho_M)$ a finite dimensional right $C$-comodule. Then $\End(M^C)$ - the set of comodule endomorphisms of $M$ (equivalently, endomorphisms of $M$ as left $C^*$-module) is a finite dimensional algebra considered with the opposite composition as multiplication. Considering $\End(M^C)$ as acting on $M$ on the right, $M$ becomes a $C^*$-$\End(M^C)$ bimodule. Denote $(A_M,\delta_M,e_M)$ the finite dimensional coalgebra dual to $\End(M^C)$; then it is easy to see that $M$ is an $A_M$-$C$ bicomodule, with the induced left $A_M$-comodule structure coming from the structure of a right $\End(M^C)$-module (this holds since there is an equivalence of categories $\Mm_{\End(M^C)}\simeq {}^{A_M}\Mm$ since $A_M$ is finite dimensional). Let $r_M:M\rightarrow A_M\otimes M$ be the left $A_M$-comultiplication of $M$. We will use the following Sweedler $\sigma$-notation:
\begin{eqnarray*}
\rho_M(m) & = & m_0\otimes m_1\in M\otimes C {\rm \,for\,} m\in M\\
r_M(m) & = & m_{(-1)}\otimes m_{(0)}\in A_M\otimes M {\rm \,for\,}m\in M\\
\Delta(c) & = & c_1\otimes c_2\in C\otimes C {\rm \,for\,} c\in C\\
\delta_M(a) & = & a_{(1)}\otimes a_{(2)}\in A_M\otimes A_M {\rm \,for\,} a\in A_M
\end{eqnarray*}
Then the compatibility relation between the left $A_M$-comodule and the right $C$-comodule structures of $M$ is written in $\sigma$-notation as
$$(*)\,\,\,\,\,\,m_{(-1)}\otimes m_{(0)0}\otimes m_{(0)1}=m_{0(-1)}\otimes m_{0(0)}\otimes m_{1}$$
We now proceed with the first step of our construction. Let $\Rr(C)$ be a {\bf set} of representatives for the isomorphism types of finite dimensional right $C$-comodules. With the above notations, let 
$$C'=(\bigoplus\limits_{M\in\Rr(C)}A_M)\oplus(\bigoplus\limits_{M\in\Rr(C)}M)\oplus C$$
and define $\delta:C'\rightarrow C'\otimes C'$ and $e:C'\rightarrow K$ by
\begin{eqnarray*}
\delta(a) & = & \delta_M(a)=a_{(1)}\otimes a_{(2)}\in A_M\otimes A_M\subseteq C'\otimes C' {\rm \,for\,} a\in A_M, M\in\Rr(C) \label{eq1}\\
\delta(m) & = & r_M(m)+\rho_M(m)=m_{(-1)}\otimes m_{(0)}+m_0\otimes m_1\in A_M\otimes M+M\otimes C\subseteq C'\otimes C' \\
(E1)\,\,\,\,\,\,\,\,\,\,\,& & {\rm \,for\,} m\in M, M\in \Rr(C) \label{eq2}\\
\delta(c) & = & \Delta(c)=c_1\otimes c_2\in C\otimes C\subseteq C'\otimes C' {\rm \,for\,} c\in C \label{eq3}
\end{eqnarray*}
(everything is understood as belonging to the appropriate - corresponding component of the tensor product $C'\otimes C'$)
\begin{eqnarray*}
e(a) & = & e_M(a), {\rm \,for\,} a\in A_M, M\in\Rr(C) \label{eq4}\\
(E2)\;\;\;\;\;\;\;\;\;\;\;\;\;\;\;\;\;\;\;\;\;\;\;\;\;\;\;e(m) & = & 0, {\rm \,for\,} m\in M, M\in\Rr(C) \;\;\;\;\;\;\;\;\;\;\;\;\;\;\;\;\;\;\;\;\;\;\;\;\;\;\;\;\;\;\;\;\;\;\;\;\;\;\;\;\;\;\;\;\;\;\;\;\;\;\;\;\;\\
e(c) & = & \varepsilon(c), c\in C
\end{eqnarray*}
It is not difficult to see that $(C',\delta,e)$ is a coalgebra. For example, for $m\in M$, $M\in\Rr(C)$
\begin{eqnarray*}
(\delta\otimes {\rm Id})\delta(m) & = & (\delta\otimes{\rm Id})(m_{(-1)}\otimes m_{(0)}+m_0\otimes m_1)\\
& = & m_{(-1)(1)}\otimes m_{(-1)2}\otimes m_{(0)}+m_{0(-1)}\otimes m_{0(0)}\otimes m_0+m_{00}\otimes m_{01}\otimes m_1
\end{eqnarray*}
and
\begin{eqnarray*}
({\rm Id}\otimes \delta)\delta(m) & = & ({\rm Id}\otimes \delta)(m_{(-1)}\otimes m_{(0)}+m_{0}\otimes m_{1})\\
& = & m_{(-1)}\otimes m_{(0)(-1)}\otimes m_{(0)(0)}+m_{(-1)}\otimes m_{(0)0}\otimes m_{(0)1}+m_0\otimes m_{11}\otimes m_{12}
\end{eqnarray*}
and here the first, second and third terms are equal respectively because of the coassociativity property of $M$ as left $A_M$-comodule, the compatibility from (*) and the coassociativity property of $M$ as right $C$-comodule. Also, we have $(e\otimes {\rm Id})\delta(m)=(e\otimes {\rm Id})(m_{(-1)}\otimes m_{(0)}+m_{0}\otimes m_{1})=e(m_{(-1)})\otimes m_{(0)}+e(m_{0})\otimes m_{1}=1\otimes e_M(m_{(-1)})m_{(0)}=1\otimes m$ etc.\\
For $(M,\rho_M)\in \Rr(C)$, since $C\subseteq C'$ is an inclusion of coalgebras, $M$ has an induced right $C'$-comodule structure by $\rho:M\rightarrow M\otimes C\subseteq M\otimes C'$ (the "co-restriction of scalars").

\begin{proposition}\label{6}
(i) Let $X(C)=(\bigoplus\limits_{M\in\Rr(C)}A_M)\oplus(\bigoplus\limits_{M\in\Rr(C)}M)$. Then $X(C)$ is a right $C'$-subcomodule of $C'$ and $C\oplus X(C)=C'$ as right $C'$-comodules.\\
(ii) If $M\in\Rr(C)$ and $Z_M=(\bigoplus\limits_{N\in\Rr(C)}A_N)\oplus(\bigoplus\limits_{N\in\Rr(C)\setminus \{M\}}N)\oplus C=A_M\oplus (\bigoplus\limits_{N\in\Rr(C)\setminus\{M\}}A_N\oplus N)\oplus C$, then $Z_M$ is a right $C'$-subcomodule of $C'$ and $C'/Z_M\simeq M$ as right $C'$-comodules.
\end{proposition}
\begin{proof}
Using the relations defining $\delta$, we have $\delta(X(C))\subseteq X(C)\otimes C'$. Thus (i) follows; for (ii), let $p:C'=M\oplus Z_M\rightarrow M$ be the projection. We have $\delta(Z_M)\subseteq \bigoplus\limits_{N\in\Rr(C)}(A_N\otimes A_N)\oplus\bigoplus\limits_{N\in\Rr(C)\setminus\{M\}}(A_N\otimes N+N\otimes C)\oplus C\subseteq  Z_M\otimes C'$. Then for $c'=m+z\in C'$, $m\in M$, $z\in Z_M$, we have $(p\otimes Id_{C'})\delta(z)=0$ and so 
\begin{eqnarray*}
(p\otimes {\rm Id_{C'}})\delta(m+z) & = & (p\otimes {\rm Id_{C'}})(m_{(-1)}\otimes m_{(0)}+m_0\otimes m_1)\\
& = & p(m_0)\otimes m_1=m_0\otimes m_1\\
& = & p(m+z)_0\otimes p(m+z)_1=(\rho_M\circ p)(m)
\end{eqnarray*}
so $p$ is a morphism of right $C'$-comodules. Since $p=\Ker(p)=Z_M$, (ii) follows.
\end{proof}

We now proceed with the last steps of our construction. Build the coalgebras $C^{(n)}$ inductively by setting $C^{(0)}=C$ and $C^{(n+1)}=(C^{(n)})'$ for all $n$; let $\delta_n,\varepsilon_n$ be the comultiplication and counit of $C^{(n)}$. We have $C^{(n+1)}=C^{(n)}\oplus X(C^{(n)})$ as $C^{(n+1)}$-comodules by Proposition \ref{6}(i). Let 
$$C_\infty=\bigcup\limits_{n\geq 1}C^{(n)}$$
as a coalgebra with $\delta_\infty$, $\varepsilon_\infty$ defined as $\delta_\infty\vert_{C^{(n)}}=\delta_n$, $\varepsilon_\infty\vert_{C^{(n)}}=\varepsilon_n$. We also note that $\delta_\infty(X(C^{(n)}))=\delta_n(X(C^{(n)}))\subseteq X(C^{(n)})\otimes C^{(n)}\subseteq X(C^{(n)})\otimes C_\infty$ so each $X(C^{(n)})$ is a right $C_\infty$-subcomodule in $C_\infty$ and we therefore actually have
\begin{equation}\label{eq}
C_\infty=C\oplus \bigoplus\limits_{n\geq 1}X(C^{(n)})=C^{(n)}\oplus\bigoplus\limits_{k\geq n+1}X(C^{(k)})
\end{equation}
as right $C_\infty$-comodules. We can now conclude our

\begin{theorem}\label{t2}
The coalgebra $C_\infty$ has the property that any finite dimensional right $C_\infty$-comodule is a quotient of $C_\infty$. Consequently, $C_\infty$ satisfies the right generating condition.
\end{theorem}
\begin{proof}
If $(N,\rho_N)\in\Mm^{C_\infty}$ is finite dimensional, then the coalgebra $D$ associated to $N$ is finite dimensional (see \cite[Proposition 2.5.3]{DNR}; this follows since the image of $\rho_N$ is finite dimensional and then the second tensor components in $N\otimes C_\infty$ from a basis of $\rho_N(N)$ span a finite dimensional coalgebra). If $d_1,\dots,d_k$ is a basis of $D$, then there is an $n$ such that $d_1,\dots,d_k\in C^{(n)}$ i.e. $D\subseteq C^{(n)}$ so $\rho_N:N\rightarrow N\otimes D\subseteq N\otimes C^{(n)}\subseteq N\otimes C_\infty$. Thus $N$ has an induced right $C^{(n)}$-comodule structure and so $\exists$ $M\in\Rr(C^{(n)})$ such that $N\simeq M$ as $C^{(n)}$-comodules. Thus, by proposition \ref{6}(ii), there is an epimorphism $C^{(n)}\rightarrow N\rightarrow 0$ of right $C^{(n)}$-comodules. Then this is also an epimorphism in $\Mm^{C_\infty}$; by equation (\ref{eq}) $C^{(n)}$ is a quotient of $C_\infty$ (in $\Mm^{C_\infty}$) and consequently $N$ must be a quotient of $C_\infty$ as right $C_\infty$-comodules. Since any right $C_\infty$-comodule is the sum of its finite dimensional subcomodules, the statement follows.
\end{proof}

\begin{example}\label{e1}
Let $C$ be a coalgebra which is not right semiperfect. Then $C_\infty$ is not right semiperfect either, since a subcoalgebra of a semiperfect coalgebra is semiperfect (see \cite[Corollary 3.2.11]{DNR}). Then $C_\infty$ cannot be right QcF, since right QcF coalgebras are right semiperfect (see \cite[Corollary 3.3.6]{DNR}; see also \cite{NT1}), so $C_\infty$ is not left projective by Theorem \ref{t0}. But still, $C_\infty$ is a generator for the category of right $C_\infty$-comodules.
\end{example}

\begin{remark}
It is also possible for a coalgebra to be right generator and not be projective to the right; indeed, just take a coalgebra $C$ which is right QcF but not left QcF; then $C$ generates $\Mm^C$ but $C^C$ is not projective since it is not left QcF (such a coalgebra exists, e.g. see \cite[Example 3.3.7]{DNR}).
\end{remark}

\begin{example}\label{e2}
Let $A$ be the algebra dual to the coalgebra $C$ of \cite[Example 3.3.7 and Example 3.2.8]{DNR}. This is left QcF and not right QcF, and $C_0$ is not finite dimensional, and thus $C^*$ is not semilocal ($C^*/J\simeq C_0^*$). By \cite[Corollary 3.3.9]{DNR}, $C^*$ is right selfinjective, and it cannot be a right cogenerator since it is not semilocal. 
\end{example}

\subsection*{Another construction}

In the following we build another example of a coalgebra with the right generating condition without being right QcF, but this will be a colocal coalgebra, that is, a coalgebra whose coradical is a simple (even 1-dimensional) coalgebra. Thus, this will show that another important condition in the theory of coalgebras, the condition that the coradical is finite dimensional, is not enough to have that the right generating condition implies that the coalgebra is right QcF.

Let $K$ be an algebraically closed field and $(C,\Delta,\varepsilon)$ be a colocal pointed $K$-coalgebra, so the coradical $C_0$ of $C$ is $C_0=Kg$, with $g$ a grouplike element: $\Delta(g)=g\otimes g$, $\varepsilon(g)=1$ ($C$ is also called connected in this case). Let $\Ll(C)$ be a {\bf set} of representatives for the {\it indecomposable} finite dimensional right $C$-comodules. Keeping the same notations as above, we note that $\End(M^C)^{op}$ is a local $K$-algebra, since $M\in\Ll(C)$ is indecomposable. Moreover, its residue field is canonically isomorphic to $K$ since it is a finite dimensional division $K$-algebra over the algeraicaly closed field $K$. Thus, $A_M$ are colocal coalgebras and there exists a unique morphism of coalgebras $\sigma_M:K\rightarrow A_M$, with $g_M=\sigma_M(1)$ being the unique grouplike of $A_M$. 

Let $(C^\sim,\delta,e)=(\bigoplus\limits_{M\in\Ll(C)}(A_M\oplus M))\oplus C$ be the coalgebra defined by the same relations (E1) and (E2) as $C'$ above; let $I$ be generated by the elements $\{g-g_M\mid M\in \Ll(C)\}$ as a vector space (they will even form a $K$-basis). $I$ is a coideal since $\delta(g-g_M)=g\otimes g-g_M\otimes g_M=g\otimes (g-g_M)+(g-g_M)\otimes g_M$ and $e(g-g_M)=0$. Let $\Sigma^\sim=(\bigoplus\limits_{M\in\Ll(C)}Kg_M)\oplus Kg$, $\Sigma=Kg\subset C$ and $\Sigma^\vee=\Sigma^\sim/I$. Let $C^\vee=C^\sim/I$. With these notations we have 

\begin{proposition}\label{7}
$C^\vee$ is a colocal pointed coalgebra.
\end{proposition}
\begin{proof}
Denote $\sigma:K\rightarrow C$ the canonical "inclusion" morphism $\sigma(1)=g$. The dual algebra of $C^\vee$ is $(C^\vee)^*=(C^\sim/I)^*\simeq I^\perp\subseteq (C^\sim)^*=C^*\times(\prod\limits_{M\in\Ll(C)}(M^*\times A_M^*))$. Let $B=I^\perp$ which is a subalgebra of $(C^\sim)^*$ and let $J_M$ and $J$ denote the Jacobson radicals of $A_M^*$ and $C^*$ respectively. Note that $B$ consists of all families $(a;(m^*,a_M)_{M\in \Ll(C)})\in (C^\sim)^*$ with $a_M(g_M)=a(g)$, equivalently, $\sigma_M(a_M)=\sigma(a)$. If two such families add up to the identity element $1_B$ of $B$, $(a;(m^*,a_M)_{M\in \Ll(C)})+(b;(n^*,b_M)_{M\in \Ll(C)})=(1;(0,1)_{M\in \Ll(C)})$, then $a_M+b_M=1\in A_M^*$ and $a+b=1\in A$ and so $a\notin J$ or $b\notin J$ since $A$ is local, say $a\notin J=Kg^\perp$ i.e. $a(g)\neq 0$. Then $a_M(g_M)=a(g)\neq 0$ and so all $a_M$ and $a$ are invertible. Thus $(a;(m^*,a_M)_{M\in \Ll(C)})$ is invertible with inverse $(a^{-1};-(a^{-1}m^*a_M^{-1},a_M^{-1})_{M\in \Ll(C)})$. This shows that $B=I^\perp$ is local with Jacobson radical $J\times(\prod\limits_{M\in\Ll(C)}(M^*\times J_M))$ and therefore, by duality, it is not difficult to see that $C^\vee$ is colocal with coradical $\Sigma^\vee$ respectively.
\end{proof}

\begin{remark}
We can easily see that we have a morphism of coalgebras $C\hookrightarrow A_M\oplus M\oplus C\rightarrow (A_M\oplus M \oplus C)/K\cdot(g-g_M)$ which is injective; then, it is also easy to see that $C^\vee$ is the direct limit of the family of coalgebras $\{C\}\cup\{(A_M\oplus M\oplus C)/K(g-g_M)\}_{M\in\Ll(C)}$ with the above morphisms. In fact, the algebra $A\times M^*\times A_M^*$ dual to $A_M\oplus M\oplus C$ is the upper triangular "matrix" algebra with obvious multiplication:
$$
\left(
\begin{array}{ccc}
	C^* & M^* \\
	0 & A_M^*
\end{array}
\right)
$$
\end{remark}

We note that $C$ embeds in $C^\vee$ canonically as a coalgebra following the composition of morphisms $C\hookrightarrow C^\sim\rightarrow C^\sim/I=C^\vee$, since $g\notin I$ so $C\cap I=0$. This allows us to view each right $C$-comodule $M$ as a comodule over $C^\vee$ (by the "corestriction" of scalars $M\rightarrow M\otimes C \rightarrow M\otimes C^\sim\rightarrow M\otimes C^\vee$). Let $p_M:C^\sim\rightarrow C^\sim/I\rightarrow M$ be the projection. 



\begin{proposition}\label{8}
(i) $p_M$ is a morphism of right $C^\vee$-comodules.\\
(ii) Each $M\in\Ll(C)$ is a quotient of $C^\vee/\Sigma^\vee$.\\
(iii) $C/\Sigma$ is a direct summand in $C^\vee/\Sigma^\vee$ as right $C^\vee$-comodules; in fact, if we denote $X^\sim(C)=\bigoplus\limits_{N\in \Ll(C)}(A_N\oplus N)$ and $X^\vee(C)=(X^\sim(C)+I)/I$ we have an isomorphism of right $C^\vee$-comodules:
$$\frac{C^\vee}{\Sigma^\vee}\simeq\frac{C}{\Sigma}\oplus X^\vee(C)$$
\end{proposition}
\begin{proof}
(i) By Proposition \ref{6} $p_M$ it is a morphism of right $C^\sim$-comodules, and then it is also a morphism of $C^\vee$-comodules via corestriction of scalars. Since the projection $C^\sim\rightarrow C^\sim/I=C^\vee$ is a morphism of coalgebras, it is also a morphism of right $C^\vee$-comodules, and since it also factors through $I$, we get that $p_M$ is a morphism of $C^\vee$-comodules.\\
(ii) follows since $p_M$ is a morphism of right $C^\vee$-comodules which cancels on $\Sigma^\vee$. \\
(iii) Note that the coradical $\Sigma$ of $C$ is identified with $\Sigma^\vee$ by the inclusion $C\hookrightarrow C^\vee$. Also both $C$ and $X^\sim(C)$ are right $C^\sim$-subcomodules of $C^\sim$ (just as above for $X(C)$ in $C'$), and then $C$ and $X^\vee(C)$ are also $C^\vee$-subcomodules in $C^\vee$. Since we also have an isomorphism of vector spaces $\frac{C^\vee}{\Sigma^\vee}\simeq\frac{C^\sim}{\Sigma^\sim}=\bigoplus\limits_{M\in\Ll(C)}(\frac{A_M}{Kg_M}\oplus M)\oplus\frac{C}{Kg}=\frac{C}{\Sigma}\oplus X^\vee(C)=\frac{C}{\Sigma^\vee}\oplus X^\vee(C)$, the proof is finished.
\end{proof}

To end the second construction, start with an arbitrary pointed colocal coalgebra over an algebraically closed field $K$. Denote $C^{[0]}=C$ and $C^{[n+1]}=(C^{[n]})^{\vee}$ for all $n\geq 0$. Put $C_\infty^\vee=\bigcup\limits C^{[n]}$. Then we have

\begin{theorem}\label{t3}
The coalgebra $C_\infty^\vee$ is colocal and has the property that any indecomposable finite dimensional right $C_\infty^\vee$-comodule is a quotient of $C_\infty^\vee$. Consequently, $C_\infty^\vee$ has the right generating condition.
\end{theorem}
\begin{proof}
Since all the coalgebras $C^{[n]}$ are colocal, say with common coradical $\Sigma$, so will be $C_\infty^\vee$. Let $(M,\rho)$ be a finite dimensional indecomposable $C_\infty^\vee$-comodule. Then, as before, $\rho(M)\subseteq M\otimes C^{[n]}$ for some $n$, since ${\dim}M<\infty$. So $M$ has an induced structure of a right $C^{[n]}$-comodule, and by Proposition \ref{8}(ii), $M$ is a quotient of $C^{[n+1]}/\Sigma$. But Proposition \ref{8} together with the construction of $C_\infty^\vee$, ensure that $\frac{C_\infty^\vee}{\Sigma}=\frac{C^{[n+1]}}{\Sigma}\oplus\bigoplus\limits_{k\geq n+1}X^\vee(C^{[k]})$. Moreover, since each $X^\vee(C^{[k]})$ is a right $C^{[k]}$-subcomodule in $C^{[k]}/\Sigma$ which is in turn a $C_\infty^\vee$-subcomodule of $C_\infty^\vee/\Sigma$, it follows that the $X^\vee(C^{[k]})$ are actually $C_\infty^\vee$-subcomodules in $C_\infty^\vee/\Sigma$. Therefore, $C^{[n+1]}/\Sigma$ splits off in $C_\infty^\vee$, and so $C_\infty^\vee$ has $M$ as a quotient. The final conclusion follows since any finite dimensional comodule is a coproduct of finite dimensional indecomposable ones.
\end{proof}

\begin{example}\label{e3}
Let $C$ be a connected (i.e. pointed colocal) coalgebra which is not right semiperfect. Then $C_\infty^\vee$ is not right semiperfect but has the right generating condition. Then, as in example \ref{e1}, $C_\infty^\vee$ is not right QcF. More specifically, we can take $C=\CC[[X]]^o$, the divided power coalgebra over the field of complex numbers, which has a basis $c_n, n\geq 0$ with comultiplication $\Delta(c_n)=\sum\limits_{i+j=n}c_i\otimes c_j$ and counit $\varepsilon(c_n)=\delta_{0,n}$ - the Kroneker symbol.
\end{example}

\begin{remark}
We could have made $C_\infty^\vee$ to also have all its finite dimensional comodules as quotients. Indeed, for this, it is enough that at each step of the construction - in passing from $C$ to $C^\vee$ - to consider the direct sum constructing $C^\vee$ to contain countably many copies of each right $C$-comodule $M$, that is, $C^\vee=[(\bigoplus\limits_{\NN}\bigoplus\limits_{M\in \Ll(C)}(A_M\oplus M))\oplus C]/I$. Then any finite dimensional comodule will be decomposed in a direct sum of finitely many indecomposable comodules, which we will be able to generate as a quotient of only one $C^{[n]}/\Sigma$ for some $n$, since enough of these indecomposable components can be found in $C^{[n]}/\Sigma$ (in fact, it is easy to see that $X^\vee(C^{[n]})/\Sigma=\bigoplus\limits_{\NN}\bigoplus\limits_{M\in\Ll(C^{[n-1]})}A_M\oplus M$ has as quotient any finite or countable sum of $M\in\Ll(C^{[n-1]})$).
\end{remark}


\bigskip\bigskip

\begin{center}
\sc Acknowledgment
\end{center}
The author wishes to thank professor Constantin Nastasescu from University of Bucharest who initially suggested this problem and inspired this subject. He also wishes to acknowledge the great support of professor Samuel D. Schack from SUNY - Buffalo; our extensive discussions helped in many aspects of this article. The thanks extend also to the referee, for his/her very careful reading of the paper and remarks which improved the presentation of the main ideeas.

\vspace*{3mm} 
\begin{flushright}
\begin{minipage}{148mm}\sc\footnotesize

Miodrag Cristian Iovanov\\
State University of New York - Buffalo, 244 Mathematics Building, Buffalo NY, 14260-2900, USA \&\\
University of Bucharest, Faculty of Mathematics, Str.Academiei 14,
RO-70109, Bucharest, Romania \\
{\it E--mail address}: {\tt
yovanov@gmail.com; e-mail@yovanov.net}\vspace*{3mm}

\end{minipage}
\end{flushright}

\end{document}